\documentclass[11pt,oneside,english]{amsart}
\usepackage[T1]{fontenc}
\usepackage[latin9]{inputenc}
\usepackage{geometry}
\usepackage{hyperref}
\usepackage{amstext}
\usepackage{amsthm}
\usepackage{amssymb}
\usepackage{enumitem}
\usepackage{xcolor}
\allowdisplaybreaks

\makeatletter
\numberwithin{equation}{section}
\numberwithin{figure}{section}
\newenvironment{lyxlist}[1]
	{\begin{list}{}
		{\settowidth{\labelwidth}{#1}
		 \setlength{\leftmargin}{\labelwidth}
		 \addtolength{\leftmargin}{\labelsep}
		 }}
	{\end{list}}
\theoremstyle{plain}
\newtheorem{thm}{\protect\theoremname}[section]
\theoremstyle{plain}

\theoremstyle{plain}
\newtheorem{lem}[thm]{\protect\lemmaname}
\theoremstyle{plain}
\newtheorem{prop}[thm]{\protect\propositionname}
\newtheorem{definition}[thm]{\protect\definitionname}

\theoremstyle{plain}
\newtheorem{rem}[thm]{\protect\remarkname}
\newtheorem{question}[thm]{Question}

\theoremstyle{plain}

\makeatother

\usepackage{babel}
\providecommand{\corollaryname}{Corollary}
\providecommand{\lemmaname}{Lemma}
\providecommand{\propositionname}{Proposition}
\providecommand{\theoremname}{Theorem}
\providecommand{\definitionname}{Definition}
\providecommand{\remarkname}{Remark}

\newcommand{\ddbar}{\sqrt{-1}\partial\bar\partial}

\begin{document}
\title{Remarks on Singular K\"ahler-Einstein Metrics}

\author[M. Hallgren]{Max Hallgren}
\address{Department of Mathematics, Rutgers University, Hill Center for the Mathematical Sciences 
110 Frelinghuysen Rd.
Piscataway, NJ, USA}
\email{mh1564@scarletmail.rutgers.edu}

\author[G. Sz\'ekelyhidi]{G\'abor Sz\'ekelyhidi}
\address{Department of Mathematics, Northwestern University, Evanston,
  IL, USA}
\email{gaborsz@northwestern.edu}

\maketitle

\begin{abstract}
We study two different natural notions of singular K\"ahler-Einstein metrics on normal complex varieties. In the setting of singular Ricci flat K\"ahler cone metrics that arise as non-collapsed limits of sequences of K\"ahler-Einstein metrics or K\"ahler-Ricci flows, we show that an a priori weaker notion is equivalent to the stronger one introduced by Eyssidieux-Guedj-Zeriahi, and in particular the underlying variety has log terminal singularities in this case. Our method applies to more general singular K\"ahler-Einstein spaces as well, assuming that they define RCD spaces. 
\end{abstract}

\section{Introduction}

Suppose that $X$ is a normal K\"ahler variety. There are at least two natural notions of a singular K\"ahler-Einstein metric on $X$. On the one hand, we can consider smooth K\"ahler-Einstein metrics $\omega$ on $X^{\operatorname{reg}}$, which in a neighborhood of any point of $X$ are given as $\omega = \ddbar u$ for $u\in L^\infty$. An a priori more restrictive notion was introduced by Eyssidieux-Guedj-Zeriahi~\cite{EGZ}. Their definition requires $X$ to have log terminal singularities, which can be used to define a canonical measure $d\mu$ in the neighborhood of any $x\in X$. In terms of this measure, a singular K\"ahler-Einstein metric is given locally by $\omega = \ddbar u$ with $u\in L^\infty$ satsifying the Monge-Amp\`ere equation $(\ddbar u)^n = e^{-\lambda u}d\mu$. It is not hard to see that if $X$ has log terminal singularities, then both notions of singular K\"ahler-Einstein metrics are equivalent. The motivating question of this paper is the following. 

\begin{question}
Let $X$ be a normal K\"ahler variety. 

Suppose that $\omega$ is a smooth K\"ahler-Einstein metric on the regular set $X^{\operatorname{reg}}$, such that locally on $X$ we have $\omega = \ddbar u$ for bounded $u$. Does it follow that $X$ has log terminal singularities? 
\end{question}

We will show that the answer is affirmative under some conditions, which in turn are satisfied in  natural settings arising from blowup limits of sequences of smooth K\"ahler-Einstein metrics, or K\"ahler-Ricci flows. In order to state the main results, we make the following definition.

\begin{definition} \label{rough}
Let $X$ be a normal K\"ahler variety of dimension $n$. A rough K\"ahler-Einstein variety $(X,\omega)$ consists of a smooth K\"ahler metric $\omega$ on $X^{\operatorname{reg}}$ such that the following are satisfied:
\begin{enumerate}[label=(\roman*)]
    \item \label{condition:einstein} $Rc(\omega)=\lambda \omega$ on $X^{\operatorname{reg}}$ for some $\lambda \in \mathbb{R}$,
    \item \label{condition:bddpotentials} $\omega$ has bounded local potentials, 
    \item \label{condition:domination} $\omega$ locally dominates a smooth K\"ahler metric on $X$,
    \item \label{condition:metricompletion} the metric completion $(\hat{X},d_{\hat{X}})$ of $(X^{\operatorname{reg}},\omega)$ with the trivially extended measure $\omega^n$ is an $RCD(\lambda,2n)$-space, 
    
    \item \label{condition:epsregularity} ($\epsilon$-regularity) there exists $\epsilon>0$ such that for any $x\in X$ and $r \in (0,\epsilon]$ satisfying $\mathcal{H}^{2n}(B(x,r)) \geq (\omega_{2n}-\epsilon)r^{2n}$, we have $x\in X^{\operatorname{reg}}$.
\end{enumerate}
\end{definition}

Natural examples of rough K\"ahler-Einstein varieties include Ricci-flat K\"ahler cones which are either Gromov-Hausdorff limits of a sequence of smooth K\"ahler-Einstein manifolds, or $\mathbb{F}$-limits of a sequence of smooth K\"ahler-Ricci flows. We will show this in Section 4. Note that in several other situations the conditions (iii)--(v) hold once we have (i) and (ii), such as the settings studied in \cite{Sz24, AIM, GuoSong25}. 

Our main result is the following. 

\begin{thm} \label{QCartierTheorem} If $(X,\omega)$ is a rough K\"ahler-Einstein variety, then for any $x\in X$, the analytic germ $(X,x)$ is log terminal \cite[Definition 6.2.7]{Ishii}.
\end{thm}

\begin{rem} In particular, the analytic germ $(X,x)$ is $\mathbb{Q}$-Gorenstein \cite[Definition 6.2.1]{Ishii} so that some power of $K_X$ extends to a line bundle in a neighborhood of $x$. However, this does not imply in general that $X$ is itself $\mathbb{Q}$-Gorenstein: there exist (noncompact) normal K\"ahler varieties $X$ which are $\mathbb{Q}$-Gorenstein in a neighborhood of any point, but such that the index of $(X,x_i)$ is unbounded for some sequence $x_i \in X$.  On the other hand, if $X$ is quasiprojective, then it is $\mathbb{Q}$-Gorenstein.
\end{rem}

We are particularly interested in the case when $(X,\omega)$ is either compact or a (singular) Ricci flat K\"ahler cone. In these cases, we have the following strengthening of Theorem \ref{QCartierTheorem}.

\begin{thm} \label{KETheorem} Suppose that $(X,\omega)$ is a rough K\"ahler-Einstein variety, such that either $X$ is compact or $(X^{\operatorname{reg}},\omega)$ is a Ricci-flat cone. Then the following hold:
\begin{enumerate}[label=(\roman*)]
    \item $X$ is $\mathbb{Q}$-Gorenstein, and has log-terminal singularities. 

    \item The K\"ahler metric $\omega$ on $X^{\operatorname{reg}}$ extends to
a K\"ahler current $\omega$ on $X$ such that $(X,\omega)$ is a singular
K\"ahler-Einstein metric in the sense of \cite{EGZ}.  

    \item In the cone setting, the volume ratio of $X$ is an algebraic number, and $(X,\omega)$ is the unique Ricci-flat K\"ahler cone on $X$ with its Reeb vector field whose existence is guaranteed by \cite{collinsSasakiEinsteinMetrics2019a}.
\end{enumerate}
\end{thm}

Using Theorem \ref{KETheorem}, we answer in the affirmative a conjecture from \cite{sunBubblingKahlerEinsteinMetrics2023} (see after Conjecture 5.9), and resolve a question from \cite[Remark 1.4]{hallgrenKahlerRicciTangentFlows2023}.

\begin{thm} \label{applicationtoRF} Suppose $(X,d)$ is a Ricci-flat metric cone arising as a noncollapsed sequence of K\"ahler-Einstein manifolds or K\"ahler-Ricci flows. Then $X$ satisfies the conclusions of Theorem \ref{KETheorem}.
\end{thm}

\begin{rem} In particular, Theorem \ref{applicationtoRF} applies to tangent cones of any noncollapsed limit of K\"ahler-Einstein manifolds or K\"ahler-Ricci flows. 
\end{rem}

The proof of Theorem \ref{QCartierTheorem} relies on the construction of sections of multiples of $K_X$, which are bounded from below and above
in a neighborhood of any given point $x_{0}\in X$. We use the method of Donaldson-Sun~\cite{donaldsonGromovHausdorffLimitsKahler2014}, exploiting that the tangent cones of non-collapsed RCD spaces are metric cones (see Cheeger-Colding~\cite{cheegerStructureSpacesRicci1997} and De Philippis-Gigli~\cite{noncollapsedRCD}). The main new difficulty is that since initially $K_X$ is not assumed to define a $\mathbb{Q}$-line bundle on $X$, applying the H\"ormander $L^2$ method will only lead to a section on $X^{\operatorname{reg}}$. We then need to obtain a priori $C^0$ and $C^1$ estimates for holomorphic sections of $K_{X^{reg}}^\ell$ near singular points of $X$. 

In Section 2, we use
improved Kato inequalities and estimates derived from the RCD assumption to establish such estimates. 

In Section 3, we use the method of Donaldson-Sun to obtain
peaked almost-holomorphic sections of $L^{m}$ near any given point
$x_{0}$ when $m\gg 0$ for suitable line bundles $L$. Using our assumption that $\omega$ is smooth outside the analytic subset $X\setminus X^{\operatorname{reg}}$, we perturb these almost-holomorphic sections to holomorphic sections. These sections are shown to have approximately Gaussian norm near $x_0$ using the estimates from Section 2. Given this, we complete the proofs of Theorem \ref{QCartierTheorem} and Theorem \ref{KETheorem}. 

In Section 4, we prove Theorem \ref{applicationtoRF} by showing that conical limits of (possibly non-polarized) K\"ahler-Einstein manifolds or Ricci flows satisfy the assumptions of Theorem \ref{KETheorem}. In particular, we show that Ricci-flat cones arising as limits of Ricci flows satisfy an RCD property. 

\subsection*{Acknowledgements}
The authors thank Jian Song, Chenyang Xu, and Junsheng Zhang for helpful comments and discussions. M.H. was supported in part by NSF grant DMS-2202980 and G. Sz. was supported in part by NSF grant DMS-2203218. 

We are grateful to Song Sun, Jikang Wang, and Junsheng Zhang for sharing their interesting preprint \cite{SunWangZhang}, where they give an independent proof of some of our results by different methods.

\section{Elliptic estimates}
We assume throughout this section that $X$ is a rough K\"ahler-Einstein variety in the sense of Definition \ref{rough}. Our goal in this section is to derive $C^0$ and $C^1$ estimates for sections of $L^m$ for certain line bundles on $X^{\operatorname{reg}}$, including $L=K_{X^{\operatorname{reg}}}$. We begin with an improved $\epsilon$-regularity property which is an elementary consequence of Definition \ref{rough}. Recall that the $\epsilon$-regular set $\mathcal{R}_{\epsilon}(Y)$ of a $2n$-dimensional noncollapsed RCD space $Y$ is the set of $p\in Y$ satisfying
$$\lim_{r\to 0} \frac{\mathcal{H}^{2n}(B(p,r))}{r^{2n}}>\omega_{2n}-\epsilon,$$
where $\omega_{2n}$ is the volume of the Euclidean unit ball. 

\begin{lem} \label{improvedepsreg} Suppose $(X,\omega)$ is a rough K\"ahler-Einstein variety, with $Rc(\omega)=\lambda \omega$ for some $|\lambda|\leq 1$, and let $\epsilon>0$ be as in Definition \ref{rough}\ref{condition:epsregularity}. Then there exists $\epsilon '=\epsilon'(\epsilon,n,\lambda)>0$ such that the following hold: 
\begin{enumerate}[label=(\roman*)]
\item For any $x\in X$ and $r\in (0,\epsilon']$ with 
$$\mathcal{H}^{2n}(B(x,r))\geq (\omega_{2n}-2\epsilon')r^{2n}$$
we have $B(x,\epsilon'r)\subset \subset X^{\operatorname{reg}}$ and
$$\sup_{B(x,\epsilon'r)} |Rm|\leq \frac{1}{(\epsilon'r)^2}.$$

\item Given any sequence $x_i\in \hat{X}$ and $r_i \in (0,1]$ such that $(\hat{X},r_i^{-1}d_{\hat{X}},x_i)$ converges in the pointed Gromov-Hausdorff sense to a noncollapsed $RCD(\lambda,2n)$ space $(Y,d_Y,x_{\infty})$, the convergence is smooth on $\mathcal{R}_{\epsilon'}(Y)$ in the following sense. $\mathcal{R}_{\epsilon'}(Y)$ is an open subset of $Y$ with the structure of a smooth K\"ahler manifold $(J_Y,g_Y)$, and there is a precompact open exhaustion $(U_i)$ of $\mathcal{R}_{\epsilon'}(Y)$ along with diffeomorphisms $\psi_i :U_i \to V_i \subseteq X^{\operatorname{reg}}$ such that $\psi_i$ converge locally uniformly to the identity map on $\mathcal{R}_{\epsilon'}(Y)$ with respect to the Gromov-Hausdorff convergence, and $$\psi_i^{\ast}J\to J_Y, \qquad \psi_i^{\ast} (r_i^{-2}g_i)\to g_Y$$
in $C_{\operatorname{loc}}^{\infty}(\mathcal{R}_{\epsilon'}(Y))$, where $J$ is the complex structure on $X^{\operatorname{reg}}$. 
\end{enumerate}
\end{lem}
\begin{proof} (i) If $\epsilon '\in (0,\frac{1}{4}\epsilon)$, then for any $x\in X$ and $r\in (0,\epsilon']$ with $\mathcal{H}^{2n}(B(x,r))\geq (\omega_{2n}-2\epsilon')r^{2n}$, relative volume comparison gives
$$\mathcal{H}^{2n}(B(y,r))\geq (\omega_{2n}-\epsilon)r^{2n}$$
for all $y\in B(x,2c(\epsilon)r)$. By Definition \ref{rough} \ref{condition:epsregularity}, it follows that $B(x,2c(\epsilon)r) \subset X^{\operatorname{reg}}$, so using Definition \ref{rough} \ref{condition:einstein}, the claim follows from Anderson's $\epsilon$-regularity \cite[Theorem 3.2]{andersoneps}.

(ii) Given $y \in \mathcal{R}_{\epsilon'}(Y)$, there exists $r=r(y)\in (0,\epsilon']$ such that $\mathcal{H}^{2n}(B(y,r))>(\omega_{2n}-\epsilon')r^{2n}$. Given any sequence $y_i \in X$ converging to $y$ with respect to the Gromov-Hausdorff convergence $(\hat{X},r_i^{-1}d_{\hat{X}},x_i)\to (Y,d_Y,x_{\infty})$, Colding's volume convergence theorem gives $$\mathcal{H}^{2n}(B(y_i,r))>(\omega_{2n}-\epsilon
)r^{2n}$$
for sufficiently large $i\in \mathbb{N}$. By (i) and the Cheeger-Gromov compactness theorem, it follows that some neighborhood $B_y$ of $y$ is isometric to a smooth K\"ahler manifold $(J_y,g_y)$, and that we can pass to a subsequence to obtain diffeomorphisms $\psi_{i,y}:B_y \to X^{\operatorname{reg}}$ converging uniformly to the identity map of $B_y$ (with respect to the pointed Gromov-Hausdorff convergence), such that $(\psi_{i,y}^{\ast}J,\psi_{i,y}^{\ast}(r_i^{-2}g))\to (J_y,g_y)$ in $C_{\operatorname{loc}}^{\infty}(B_y)$. A standard construction then allows us to patch together these diffeomorphisms $\psi_{i,y}$ to obtain the desired global diffeomorphisms $\psi_i$. 
\end{proof}

In order to show that the singular set $\hat{X} \setminus X^{\operatorname{reg}}$ has singularities of codimension $>1$, we need to adapt a cutoff function construction \cite[Lemma 3.7]{songcurrents} of Sturm to the local setting.

\begin{lem} \label{sturmcutofflemma}
For any $x_{0}\in\hat{X}$, there exists $r\in(0,1]$ and $C<\infty$
such that the following holds. For any compact subset $\mathcal{K}\subseteq X^{\operatorname{reg}}$,
there exists $\rho\in C_{c}^{\infty}(X^{\operatorname{reg}},[0,1])$
such that $\rho|_{B(x_{0},r)\cap\mathcal{K}}\equiv1$, $\mbox{supp}(\rho)\subseteq B(x_0,2r)$, and $\int_{B(x_{0},r)\cap X^{\operatorname{reg}}}|\nabla\rho|^{2}\omega^{n}\leq C$. 
\end{lem}

\begin{proof} By Definition \ref{rough}\ref{condition:metricompletion} and \cite{MondinoNaber}[Lemma 3.1], there exists a cutoff function $\phi$ on $\widehat{X}$ with $\phi|_{B(x_{0},r)}\equiv1$,
$\operatorname{supp}(\phi)\subset\subset B(x_{0},2r)$, and $r|\nabla\phi|+r^{2}|\Delta\phi|\leq C(n,\lambda)$ on $X^{\operatorname{reg}}$. By Definition \ref{rough}\ref{condition:domination}, we can choose $r>0$ sufficiently small so that 
\[
B(x_{0},2r)\setminus X^{\operatorname{reg}}=\{x\in B(x_{0},2r);f_{1}(x)=\cdots=f_{N}(x)=0\}
\]
for some $f_{1},...,f_{N}\in\mathcal{O}_{X}(B(x,2r))$. Let $F\in C^{\infty}([0,\infty),[0,1])$
be a smooth cutoff function satisfying $F|_{[0,\frac{1}{2}]}\equiv1$
and $F|_{[1,\infty)}\equiv0$. Define

\[
\eta_{i,\epsilon}:=\max\left(\log|f_{i}|^{2},\log\epsilon\right),
\]
\[
\rho_{i,\epsilon}:=\phi\cdot F\left(\frac{\eta_{i,\epsilon}}{\log\epsilon}\right),
\]
so that $\log\epsilon\leq\eta_{i,\epsilon}\leq0$, $\sqrt{-1}\partial\overline{\partial}\eta_{i,\epsilon}\geq0$ in the sense of currents, $\rho_{i,\epsilon}|_{B(x_{0},r)\cap\{|f_{i}|\geq\epsilon^{\frac{1}{4}}\}}\equiv1$, 
and $\operatorname{supp}(\rho_{i,\epsilon})\subseteq B(x_{0},2r)\cap\{|f_{i}|\geq\epsilon^{\frac{1}{2}}\}$.
By rescaling $f_{1},...,f_{N}$, we may also assume $\eta_{i,\epsilon}\leq0$
for all $\epsilon\in(0,1]$. We integrate by parts to estimate

\begin{align*}
\int_{X^{\operatorname{reg}}}\phi^{2}\sqrt{-1}\partial\eta_{i,\epsilon}\wedge\overline{\partial}\eta_{i,\epsilon}\wedge\omega^{n-1}= & \int_{X^{\operatorname{reg}}}\phi^{2}(-\eta_{i,\epsilon})\sqrt{-1}\partial\overline{\partial}\eta_{i,\epsilon}\wedge\omega^{n-1}-2\text{Re}\int_{X^{\operatorname{reg}}}\sqrt{-1}\eta_{i,\epsilon}\partial\phi\wedge\phi\overline{\partial}\eta_{i,\epsilon}\wedge\omega^{n-1}\\
\leq & \int_{X^{\operatorname{reg}}}(-\eta_{i,\epsilon})\phi^{2}\sqrt{-1}\partial\overline{\partial}\eta_{i,\epsilon}\wedge\omega^{n-1}+\frac{1}{2}\int_{X^{\operatorname{reg}}}\phi^{2}\sqrt{-1}\partial\eta_{i,\epsilon}\wedge\overline{\partial}\eta_{i,\epsilon}\wedge\omega^{n-1}\\
 & +2\int_{X^{\operatorname{reg}}}\eta_{i,\epsilon}^{2}\sqrt{-1}\partial\phi\wedge\overline{\partial}\phi\wedge\omega^{n-1},
\end{align*}
so that
\begin{align*}
\int_{X^{\operatorname{reg}}}\sqrt{-1}\partial\rho_{i,\epsilon}\wedge\overline{\partial}\rho_{i,\epsilon}\wedge\omega^{n-1} \hspace{-30 mm}\\ = & \int_{X^{\operatorname{reg}}}\sqrt{-1}\left(F\left(\frac{\eta_{i,\epsilon}}{\log\epsilon}\right)\partial\phi+ \frac{1}{\log \epsilon}F'\left(\frac{\eta_{i,\epsilon}}{\log\epsilon}\right)\phi\partial\eta_{i,\epsilon}\right) \\
&\qquad \wedge\left(F\left(\frac{\eta_{i,\epsilon}}{\log\epsilon}\right)\overline{\partial}\phi+\frac{1}{\log\epsilon}F'\left(\frac{\eta_{i,\epsilon}}{\log\epsilon}\right)\phi\overline{\partial}\eta_{i,\epsilon}\right)\wedge\omega^{n-1}\\
\leq & 2\int_{X^{\operatorname{reg}}}\sqrt{-1}\partial\phi\wedge\overline{\partial}\phi\wedge\omega^{n-1}+\frac{C}{|\log\epsilon|^{2}}\int_{X^{\operatorname{reg}}}\phi^{2}\sqrt{-1}\partial\eta_{i,\epsilon}\wedge\overline{\partial}\eta_{i,\epsilon}\wedge\omega^{n-1}\\
\leq & C(n,\lambda)\int_{B(x_{0},2r)\cap X^{\operatorname{reg}}}\omega^{n}+\frac{C}{|\log\epsilon|^{2}}\int_{X^{\operatorname{reg}}}(-\eta_{i,\epsilon})\phi^{2}\sqrt{-1}\partial\overline{\partial}\eta_{i,\epsilon}\wedge\omega^{n-1}\\
\leq & C(n,\lambda)+\frac{C}{|\log\epsilon|}\int_{X^{\operatorname{reg}}}\phi^{2}\sqrt{-1}\partial\overline{\partial}\eta_{i,\epsilon}\wedge\omega^{n-1}\\
\leq & C(n,\lambda)+\frac{C}{|\log\epsilon|}\int_{X^{\operatorname{reg}}}\eta_{i,\epsilon} (|\nabla \phi|^2 +\phi |\Delta \phi|)\omega^n\\
\leq & C(n,\lambda).
\end{align*}
Choose $\sigma\in C^{\infty}([0,\infty),[0,1])$ such
that $\sigma|_{[0,\frac{1}{4}]}\equiv0$ and $\sigma|_{[\frac{3}{4},\infty)}\equiv1$,
choose $\epsilon>0$ such that $B(x_{0},r)\cap\mathcal{K}\subseteq\bigcup_{i=1}^{N}\{|f_{i}|\geq\epsilon^{\frac{1}{4}}\},$
and set
\[
\rho:=\sigma\left(\sum_{i=1}^{N}\rho_{i,\epsilon}\right).
\]
Then $\int_{B(x_0,r)\cap X^{\operatorname{reg}}}|\nabla\rho|^{2}\omega^{n}\leq C(n,\lambda)$,
\[
\operatorname{supp}(\rho)\subseteq B(x_{0},2r)\cap\bigcup_{i=1}^{N}\{|f_{i}|\geq\epsilon^{\frac{1}{2}}\}\subseteq X^{\operatorname{reg}},
\]
and $\rho=1$ on 
\[ \{ \rho = 1\}\supseteq 
\bigcup_{i=1}^{N}\{\rho_{i,\epsilon}=1\}\supseteq B(x_{0},r)\cap\bigcup_{i=1}^{N}\{|f_{i}|\geq\epsilon^{\frac{1}{4}}\}\supseteq B(x_{0},r)\cap\mathcal{K}.
\]
\end{proof}

The following is an essential ingredient for establishing the $C^0$ and $C^1$ estimates. 
\begin{prop} \label{codim4}
    The singular sets of $\hat{X}$ and its tangent cones have singularities of Hausdorff codimension at least 4. In particular,
    $\hat{X}\setminus X^{\operatorname{reg}}$ has Hausdorff codimension at least 4. 
\end{prop}
\begin{proof}
    Given $x_0 \in X$, choose $r>0$ such that Lemma \ref{sturmcutofflemma} applies. Choose a precompact exhaustion $(\mathcal{K}_j)$ of $\hat{X}\setminus X^{\operatorname{reg}}$, so that Lemma \ref{sturmcutofflemma} gives $\phi_j \in C_c^{\infty}(X^{\operatorname{reg}},[0,1])$ satisfying $\rho_j|_{B(x_0,r)\cap \mathcal{K}_j}\equiv 1$, and $\int_{B(x_0,r)\cap X^{\operatorname{reg}}}|\nabla \rho_j|^{2}dg \leq C(n,\lambda)$. Then H\"older's inequality gives
    $$\limsup_{j\to \infty} \int_{B(x_0,r)\cap X^{\operatorname{reg}}}|\nabla \rho_j|^{\frac{3}{2}}\omega^n\leq \limsup_{j\to \infty}\left( \int_{B(x_0,r)\cap X^{\operatorname{reg}}}|\nabla \rho_j|^{2}\omega^n \right)^{\frac{3}{4}} \left( \int_{B(x_0,r)\setminus \mathcal{K}_j} \omega^n\right)^{\frac{1}{4}}=0.$$
We may therefore pass to a subsequence to ensure that 
$$\int_{B(x_0,r)\cap X^{\operatorname{reg}}} (|\nabla \rho_j|^{\frac{3}{2}} + (1-\rho_j)^{\frac{3}{2}}) \omega^n \leq 2^{-j}.$$
Set $\psi:= \sum_{j=1}^{\infty} (1-\rho_j) \in W^{1,\frac{3}{2}}(B(x_0,r))$, so that $\psi<\infty$ on $X^{\operatorname{reg}}$, whereas 
\begin{equation} \label{eq:usedincodim2} \lim_{x\to B(x_0,r)\setminus X^{\operatorname{reg}}}\psi(x)=\infty.\end{equation} 
Because $\hat{X}$ is an $RCD(\lambda,2n)$ space, it satisfies a Poincar\'e inequality by \cite[Theorem 1]{poincare} and \cite[p.970]{lipdensity}, so we can argue as in \cite[claim in proof of Theorem 4.17]{EvansGariepy} to conclude from \eqref{eq:usedincodim2} that for any $x\in B(x_0,r)\setminus X^{\operatorname{reg}}$, we have 
\begin{equation} \label{eq:sdensity} \limsup_{s\searrow 0} \frac{1}{s^{2n-\frac{5}{4}}}\int_{B(x,s)}|\nabla \psi|^{\frac{3}{2}} \omega^n=\infty.\end{equation}
Because $\hat{X}$ satisfies the volume doubling property, \cite[proof of Theorem 2.10]{EvansGariepy}, \eqref{eq:sdensity}, and $\psi \in W^{1,\frac{3}{2}}(B(x_0,r))$ yield $\mathcal{H}^{2n-\frac{5}{4}}(\hat{X}\setminus X^{\operatorname{reg}})=0$. By applying Bru\`e-Naber-Semola~\cite[Theorem 1.2]{BrueNaberSemola} as in \cite[Proposition 10]{Sz24}, it follows that $\hat{X}$ can not admit any iterated tangent cones of the form $\mathbb{R}^{2n-1}\times [0,\infty)$, and in particular, the Hausdorff co-dimension of $\hat{X}\setminus X^{\operatorname{reg}}$ is at least 2.

    We can also rule out (iterated) tangent cones of the form $\mathbb{R}^{2n-2}\times C$ and $\mathbb{R}^{2n-3}\times C$ for cones $C$ without further lines splitting. This can be done by following the arguments in \cite[Propositions 27, 28]{Sz24}, as we now explain. Suppose some iterated tangent cone of $\hat{X}$ at $x$ is of the form $\mathbb{R}^{2n-2}\times C(\mathbb{S}_{\gamma}^1)$. This means there are $x_j \in \widehat{X}$ with $x_j \to x$, and $k_j \in \mathbb{N}$ with $\lim_{j\to \infty}k_j = \infty$ such that $(\hat{X},k_j^{\frac{1}{2}} d_{\hat{X}},x_j)$ converge in the pointed Gromov-Hausdorff sense to $\mathbb{R}^{2n-2}\times C(\mathbb{S}_{\gamma}^1)$ as $j\to \infty$. Using Definition \ref{rough}\ref{condition:bddpotentials}, we let $U$ be a Stein neighborhood of $x$ in $X$ on which $\omega=\ddbar \varphi$, so that $(L,h_j):=(\mathcal{O}_{X^{\operatorname{reg}}},e^{-k_j \varphi})$ is a polarization of $(X^{\operatorname{reg}}\cap U,\omega_j)$, where $\omega_j:=k_j \omega$. By Definition \ref{rough}\ref{condition:domination}, any sufficiently small ball in $(\hat{X}, d_{\hat{X}})$ is contained in such a Stein neighborhood $U$. We can thus argue as in \cite{Sz24}[Proposition 27], using the arguments of \cite[Proposition 9 and Section 2.5]{CDSII} as well as \cite[Theorem 0.2]{demaillyEstimationsL2Pour1982}, in order to guarantee that $\gamma=2\pi$ assuming we can prove the required $C^0$ and $C^1$ estimates for $L^2$ holomorphic functions on $X^{\operatorname{reg}}$. Any holomorphic function $u$ on $U\cap X^{\operatorname{reg}}$ extends to a holomorphic (and thus locally bounded) function on $U$ since $X$ is normal. Because $u$ is harmonic on $U$ in the sense of distributions (c.f. \cite[Lemma 11]{Sz24}), we can apply \cite[Theorem 1.1]{jiangharmonic} to conclude that $u$ is locally Lipschitz on $U$. On $U\cap X^{\operatorname{reg}}$, we have $\Delta_{\omega_j}|u|_{h_j} \geq -C|u|_{h_j}$ and $\Delta_{\omega_j} |\nabla^{h_j} u|_{h_j} \geq -C|\nabla^{h_j} u|_{h_j}$, where $C>0$ are independent of $j\in \mathbb{N}$. Because we have already shown these quantities are both locally bounded, the desired estimates follow (c.f. \cite[Proposition 19]{Sz24}). Thus any (iterated) tangent cone of the form $\mathbb{R}^{2n-2}\times C$ is actually $\mathbb{R}^{2n}$. The argument of \cite[Proposition 28]{Sz24} then shows that any (iterated) tangent cone of the form $\mathbb{R}^{2n-3}$ is actually $\mathbb{R}^{2n}$. By the definition of a rough K\"ahler-Einstein variety any point in $\hat{X}\setminus X^{\operatorname{reg}}$ is in the metric singular set of $\hat{X}$, so the Hausdorff dimension bound for $\hat{X}\setminus X^{\operatorname{reg}}$ follows from the Hausdorff dimension bounds of De Philippis-Gigli~\cite[Theorem 1.8]{noncollapsedRCD}.
\end{proof}

We now construct the cutoff functions that we will use to prove elliptic estimates on $X^{\operatorname{reg}}$. First recall from Mondino-Naber~\cite[Lemma 3.1]{MondinoNaber} that for any $r$-ball $B(x,r)\subset \hat{X}$, with $r \in (0,10)$, we have a Lipschitz function $\phi_r$ that satisfies $\phi_r = 1$ on $B(x,r)$, $\mathrm{supp}(\phi_r)\subset B(x,2r)$ and 
\begin{equation} \label{phicutoff} r^2 |\Delta \phi_r| + r |\nabla \phi_r| < C, \end{equation}
for a constant $C$ depending on $n, \lambda$.  

Using these cutoff functions, together with the fact that $\hat{X}\setminus X^{\operatorname{reg}}$ is a closed subset with codimension at least 4, we can argue similarly to Donaldson-Sun~\cite[Proposition 3.5]{donaldsonGromovHausdorffLimitsKahler2014} to construct cutoff functions $\eta_\epsilon$ as follows.
\begin{lem} \label{cutofflemma}

    There exist functions $\eta_\epsilon \in C_c^{\infty}(X^{\operatorname{reg}})$ such that for any compact subset $\mathcal{K}\subset X^{\operatorname{reg}}$ we have $\eta_\epsilon|_{\mathcal{K}} = 1$ for sufficiently small $\epsilon$. In addition for any $R, \sigma > 0$ we can arrange that 
\begin{equation} \label{etacutoff}
\lim_{\epsilon\searrow0}\int_{B(y_0,R)\cap X^{\operatorname{reg}}}(|\nabla\eta_{\epsilon}|^{4-\sigma}+|\Delta\eta_{\epsilon}|^{2-\sigma})\omega^n=0,
\end{equation}
for a basepoint $y_0$. 

\end{lem}
\begin{proof}

Using that $\Sigma = \hat{X}\setminus X^{reg}$ is closed and has Hausdorff codimension at least four, it follows that for any $\delta > 0$ we can find a cover of $\Sigma\cap B(y_0, \delta^{-1})$ with finitely many balls $B(x_i, r_i/2)$ such that
\[ \sum_i r_i^{2n-4+2\sigma} < \delta.\]
By the Vitali covering lemma we can assume that the balls $B(x_i, r_i/10)$ are disjoint. We define the function $f = \sum_i f_i$, where $f_i = \phi_{r_i}$ for the cutoff functions $\phi_{r_i}$ as above, supported on $B(x_i, 2r_i)$. Let $\Phi(t)$ be a smooth function such that $\Phi(0)=0$, $\Phi(t)=1$ for $t > 9/10$, and $|\Phi'(t)|, |\Phi''(t)|\leq 10$ for all $t$. Then define $\eta(x) = \Phi(f(x))$. We have
\[ \begin{aligned} |\nabla\eta(x)| &\leq 10 |\nabla f(x)|, \\
|\Delta\eta(x)| &\leq 10 |\nabla f(x)|^2 + 10 |\Delta f(x)|. \end{aligned} \]
Therefore it is enough to estimate the integral of $|\nabla f|^{4-\sigma} + |\Delta f|^{2-\sigma}$. 

Let us decompose the index set of the balls into the subsets
\[ I_\alpha = \{i\,:\, 2^{-\alpha-1} \leq r_i < 2^{-\alpha}\},\]
for integers $\alpha \geq 0$. 
By the volume doubling property of $\hat{X}$, there exists $N=N(n,\lambda)$ 
such that if $j\in I_\alpha$, then for any fixed $\beta \leq \alpha$ there are at most $N$ balls $B_i$ with $i\in I_\beta$ intersecting $B_j$. Consider a ball $B_j$, with $j\in I_\alpha$. Let us denote by $B_j' \subset B_j$ the set $x\in B_j$ such that for all $i\in I_{\beta}$ with $\beta>\alpha$, we have $x\notin B_i$. 
If $x\in B_j'$, then for each $\beta \leq \alpha$ there are at most $N$ balls $B_i$ with $i\in I_\beta$ and $x\in B_i$. It follows from this that for any $x\in B_j'$, we have
$$|\nabla f(x)|\leq C(n,\lambda)\sum_{\beta=0}^{\alpha} \sum_{i\in I_{\beta}} r_i^{-1}\leq C(n,\lambda)r_j^{-1}2^{-\alpha} \sum_{\beta=0}^{\alpha} N2^{\beta} \leq C(n,\lambda)r_j^{-1},$$
and similarly $|\Delta f(x)| \leq C(n,\lambda)r_j^{-2}.$
We therefore have
\[ \int_{B_j'} (|\nabla f|^{4-\sigma} + |\Delta f|^{2-\sigma}) \,\omega^n \leq C(n,\lambda)r_j^{2n} r_j^{2\sigma-4}.\]
Given any $x\in B(y_0,\delta^{-1})$, there is a unique $\alpha \in \mathbb{N}$ such that $x\in B_i'$ for some $i\in I_{\alpha}$ (and there are at most $N$ distinct $i\in I_{\alpha}$ satisfying $x\in B_i'$). Summing over $j$, it follows that
\[ \int_{B(y_0, \delta^{-1})\cap X^{\operatorname{reg}}} (|\nabla f|^{4-\sigma} + |\Delta f|^{2-\sigma}) \,\omega^n \leq C(n,\lambda)\sum_j r_j^{2n-4+2\sigma} < C(n,\lambda)\delta.\]
Moreover, because $\sup_i r_i \leq \delta^{\frac{1}{2n-4+2\sigma}}$, we have $\mbox{supp}(\eta) \subseteq B(\hat{X}\setminus X^{\operatorname{reg}},\delta^{\frac{1}{2n-4+2\sigma}})$. We may therefore choose $\eta_\epsilon$ to be defined by $1-\eta$, for $\delta = C(n,\lambda)^{-1}\epsilon$. 
\end{proof}

Because $\hat{X}$ is an RCD space, it has a well-defined heat kernel, whose properties we now recall. 
\begin{lem}
\label{heatkernel} There exists a function $K:\hat{X}\times \hat{X}\times(0,\infty)\to(0,\infty)$ such that for any compact subset $\mathcal{K} \subseteq \hat{X}$, there exists $C=C(\mathcal{K})$ such that the following hold:

$(i)$ $K$ is continuous, and $K|_{X^{\operatorname{reg}}\times X^{\operatorname{reg}}\times(0,\infty)}$
is smooth,

$(ii)$ $(\partial_{t}-\Delta_{x})K(x,y,t)=0$ for all $x,y\in X^{\operatorname{reg}}$
and $t>0$,

$(iii)$ For any Lipschitz $\psi\in C_{c}(X^{\operatorname{reg}})$, $\lim_{t\searrow0}\int_{\hat{X}}K(x,y,t)\psi(y)dg(y)=\psi(x)$
for all $x\in X^{\operatorname{reg}}$,

$(iv)$ $K(x,y,t)=K(y,x,t)$ for all $x,y\in X^{\operatorname{reg}}$ and
$t>0$,

$(v)$ $K(x,y,t)\leq\frac{C}{t^{n}}\exp\left(-\frac{d^{2}(x,y)}{Ct}\right)$
for all $x,y\in \mathcal{K}$ and $t\in (0,1]$,

$(vi)$ $|\nabla_{x} K(x,y,t)|=|\nabla_{y} K(x,y,t)|\leq\frac{C}{t^{\frac{n+1}{2}}} \exp\left(-\frac{d^{2}(x,y)}{Ct}\right)$
for all $x,y\in X^{\operatorname{reg}}\cap \mathcal{K}$ and $t\in (0,1]$. 
\end{lem}

\begin{proof} Assertions (i),(ii) are justified by the fact that that the Laplacian is strongly local, while (iii) follows from the fact that $\lim_{t\searrow 0}\int_{\hat{X}} K(x,y,t)\psi(y)dg(y)=\psi (x)$ a.e. for $\psi \in L^2(\hat{X})$, and (iv) follows from the fact that $\Delta$ is a self-adjoint densely-defined operator on $L^2(\hat{X})$. The estimates (v),(vi) follow from \cite[Theorem 1.2, Corollary 1.2]{jiangheat} and relative volume comparison. 
\end{proof}
We can use the heat kernel estimates of Lemma \ref{heatkernel} to
prove the following $C^{0}$ estimate for subsolutions of an elliptic
equation. The proof is similar to \cite[Lemma 4]{liuGromovHausdorffLimitsAhler2019}, except that we now require that $p$ is strictly larger than $2$ to make up for the lack of sharp estimates on the size of the singular set. 
\begin{lem}
\label{elliptic} Given $p>2$ and any precompact open set $B\subseteq \hat{X}$, there exists $C=C(\lambda,p,B)<\infty$ such that the following holds. Let $x_0 \in \hat{X}$ and $r\in (0,1]$ be such that $B(x_0,5r) \subseteq B$, and suppose $v:B(x_0,5r)\cap X^{\operatorname{reg}}\to [0,\infty)$ is Lipschitz on compact subsets of $B(x_0,5r)\cap X^{\operatorname{reg}}$. If $\Delta v \geq -Av$ in the sense of distributions on $B(x_0,5r)\cap X^{\operatorname{reg}}$, and if $\int_{B(x_0,5r)\cap X^{\operatorname{reg}}}v^p \omega^n<\infty$, then
$$\sup_{B(x_0,r)\cap X^{\operatorname{reg}}}|v|\leq Ce^{Ar^2} \left( \frac{1}{r^{2n}} \int_{X^{\operatorname{reg}}\cap B(x_0,5r)} |v|^p \omega^n \right)^{\frac{1}{p}}.$$
\end{lem}
\begin{proof} By the discussion preceding \eqref{phicutoff}, we can choose $\phi_r:X^{\operatorname{reg}}\to [0,1]$ such that $\phi_r|_{B(x_0,2r)}\equiv1$, $\text{supp}(\phi_r)\subseteq B(x_0,4r)$, and 
$$r^2 |\Delta \phi_r|+r|\nabla \phi_r|\leq C(n,\lambda)$$
on $X^{\operatorname{reg}}$. Let $\eta_{\epsilon}$ be as in Lemma \ref{cutofflemma}, and let $K$ be as in Lemma \ref{heatkernel}, with $C=C(\overline{B})<\infty$ the constant from that lemma. Let $\langle \cdot ,\cdot \rangle$ denote the $\mathbb{C}$-bilinear extension of $g$ to $TX\otimes_{\mathbb{R}}\mathbb{C}$. We integrate by parts to get
\begin{align*}
\frac{d}{dt}\int_{X^{\operatorname{reg}}}v & (y)\phi_{r}(y)\eta_{\epsilon}(y)K(x,y,r^2-t)\omega^n(y)\\
= & -\int_{X^{\operatorname{reg}}}\left(\phi_{r}(y)\eta_{\epsilon}(y)\Delta v(y)+v(y)\eta_{\epsilon}(y)\Delta\phi_{r}(y)+v(y)\phi_{r}(y)\Delta\eta_{\epsilon}(y)\right)K(x,y,r^2-t)\omega^n(y)\\
 & -2\text{Re}\int_{X^{\operatorname{reg}}}\left(\langle\nabla\phi_{r},\overline{\nabla}\eta_{\epsilon}\rangle(y)v(y)+\eta_{\epsilon}(y)\langle\nabla\phi_{r},\overline{\nabla}v\rangle(y)+\phi_{r}(y)\langle\nabla\eta_{\epsilon},\overline{\nabla}v\rangle(y)\right)K(x,y,r^2-t)\omega^n(y)\\
= & -\int_{X^{\operatorname{reg}}}\left(\phi_{r}(y)\eta_{\epsilon}(y)\Delta v(y)+v(y)\eta_{\epsilon}(y)\Delta\phi_{r}(y)+v(y)\phi_{r}(y)\Delta\eta_{\epsilon}(y)\right)K(x,y,r^2-t)\omega^n(y)\\
 & +2\text{Re}\int_{X^{\operatorname{reg}}}\left(\langle\nabla\phi_{r},\overline{\nabla}\eta_{\epsilon}\rangle+\eta_{\epsilon}\Delta\phi_{r}+\eta_{\epsilon}\langle\nabla\phi_{r},\overline{\nabla}\log K(x,\cdot,r^2-t)\rangle\right)(y)v(y)K(x,y,r^2-t)\omega^n(y)\\
 & +2\text{Re}\int_{X^{\operatorname{reg}}}\left(\phi_{r}\Delta\eta_{\epsilon}+\phi_{r}\langle\nabla\eta_{\epsilon},\overline{\nabla}\log K(x,\cdot,r^2-t)\rangle\right)(y) v(y) K(x,y,r^2-t)\omega^n(y)\\
\leq & A \int_{X^{\operatorname{reg}}} v(y)\phi_r(y)\eta_{\epsilon}(y)K(x,y,r^2-t)\omega^n(y)\\
& + C\int_{X^{\operatorname{reg}}\cap B(x_0,4r)}\left(|\Delta\eta_{\epsilon}|+r^{-1}|\nabla\eta_{\epsilon}|+|\Delta\phi_{r}|\right)(y)v(y)K(x,y,r^2-t)\omega^n(y)\\
 & +C\int_{X^{\operatorname{reg}}\cap B(x_0,4r)}(|\nabla\phi_{r}|+|\nabla\eta_{\epsilon}|)(y)v(y)|\nabla_{y}K(x,y;r^2-t)|\omega^n(y)
\end{align*}
for any $x\in X^{\operatorname{reg}}$ and $t\in[0,r^2)$. For $x\in B(x_0,r)\cap X^{\operatorname{reg}}$
fixed, there exists $r_{0}=r_{0}(x)>0$ such that for all $\epsilon>0$
sufficiently small, we have
\[
d(\operatorname{supp}(1-\eta_{\epsilon}),x)\geq r_{0}.
\]
Letting $q:=\left(1-\frac{1}{p}\right)^{-1}\in(1,2)$, we can use Lemma \ref{heatkernel}(v) to estimate
\begin{align*}
\int_{X^{\operatorname{reg}}\cap B(x_0,4r)} |\Delta\eta_{\epsilon}(y)|v(y)K(x,y,r^2-t)\omega^n(y) \hspace{-68mm} &\\
\leq & \left(\int_{X^{\operatorname{reg}}\cap B(x_0,4r)}|\Delta\eta_{\epsilon}(y)|^{q}K^{q}(x,y,r^2-t)\omega^n(y)\right)^{\frac{1}{q}}\left(\int_{X^{\operatorname{reg}}\cap B(x_0,4r)}v(y)^{p}\omega^n(y)\right)^{\frac{1}{p}}\\
\leq & C\left(\int_{\operatorname{supp}(1-\eta_{\epsilon})\cap B(x_0,4r)}\frac{|\Delta\eta_{\epsilon}(y)|^{q}}{(r^2-t)^{qn}}\exp\left(-\frac{r_{0}^{2}}{C(r^2-t)}\right)\omega^n(y)\right)^{\frac{1}{q}}\left(\int_{X^{\operatorname{reg}}\cap B(x_0,4r)}v(y)^{p}\omega^n(y)\right)^{\frac{1}{p}}\\
\le & C(r_{0})\left(\int_{B(x_0,4r)}|\Delta\eta_{\epsilon}|^{q}\omega^n(y)\right)^{\frac{1}{q}}\left(\int_{X^{\operatorname{reg}}\cap B(x_0,4r)}v(y)^{p}\omega^n(y)\right)^{\frac{1}{p}}
\end{align*}
for all $t\in [0,r^2)$ when $\epsilon=\epsilon(x)>0$ is sufficiently small. Using \eqref{phicutoff}, we similarly have
\begin{align*}
\int_{X^{\operatorname{reg}}\cap B(x_0,4r)}|\Delta\phi_{r}(y)| & v(y)K(x,y,r^2-t)\omega^n(y)\\
\leq & \frac{C}{r^{2}}\int_{X^{\operatorname{reg}}\cap(B(x_0,4r)\setminus B(x_0,2r))}v(y)K(x,y,r^2-t)\omega^n(y)\\
\leq & \frac{C}{r^{2}}\frac{e^{-\frac{r^2}{C(r^2-t)}}}{(r^2-t)^{n}}\int_{X^{\operatorname{reg}}\cap(B(x_0,4r)\setminus B(x_0,2r))}v(y)\omega^n(y)\\
\leq & Cr^{\frac{2n}{q}-2n-2}\frac{r^{2n}}{(r^2-t)^{n}}e^{-\frac{r^2}{C(r^2-t)}}\left(\int_{X^{\operatorname{reg}}\cap B(x_0,4r)}v(y)^{p}\omega^n(y)\right)^{\frac{1}{p}}\\
\leq & \frac{C}{r^2}\left(\frac{1}{r^{2n}}\int_{X^{\operatorname{reg}}\cap B(x_0,4r)}v(y)^{p}\omega^n(y)\right)^{\frac{1}{p}}.
\end{align*}
The remaining terms can be estimated similarly, using Lemma \ref{heatkernel}(vi):
\begin{align*}
\int_{X^{\operatorname{reg}}}|\nabla\phi_{r}|(y) & v(y)|\nabla_{y}K(x,y;r^2-t)|\omega^n(y)\\
\leq & \frac{Ce^{-\frac{r^2}{C(r^2-t)}}}{r(r^2-t)^{n+\frac{1}{2}}}\int_{X^{\operatorname{reg}}\cap(B(x_0,4r)\setminus B(o,r))}v(y)\omega^n(y)\\
\leq & \frac{C}{r^2}\left(\frac{1}{r^{2n}}\int_{X^{\operatorname{reg}}\cap B(x_0,4r)}v(y)^{p}\omega^n(y)\right)^{\frac{1}{p}},
\end{align*}
\[
\int_{X^{\operatorname{reg}}}|\nabla\eta_{\epsilon}|(y)v(y)|\nabla_{y}K(x,y;r^2-t)|\omega^n(y)\leq C(r,r_{0},p)\left(\int_{B(x_0,4r)}|\nabla\eta_{\epsilon}|^{q}\omega^n(y)\right)^{\frac{1}{q}}\left(\int_{X^{\operatorname{reg}}\cap B(x_0,4r)}v(y)^{p}\omega^n(y)\right)^{\frac{1}{p}}.
\]
Integrating from $t=0$ to $t=r^2$, using Lemma \ref{heatkernel}(iii), and then combining the above estimates with \eqref{etacutoff} yields
\begin{align*}
v(x)\leq & e^{Ar^2}\int_{B(x_0,4r)\cap X^{\operatorname{reg}}}v(y)K(x,y,r^2)\omega^n(y)\\
 & +Ce^{Ar^2}\left(\frac{1}{r^{2n}}\int_{X^{\operatorname{reg}}\cap B(x_0,4r)}v(y)^{p}\omega^n(y)\right)^{\frac{1}{p}}+\Psi(\epsilon|r_{0},r).
\end{align*}
By H\"older's inequality, we have
\begin{align*} \int_{B(x_0,4r)\cap X^{\operatorname{reg}}} v(y)K(x,y,r^2)\omega^n(y) \leq \frac{C}{r^{2n}} \int_{B(x_0,4r)} v\omega^n \leq C\left(\frac{1}{r^{2n}}\int_{X^{\operatorname{reg}}\cap B(x_0,4r)}v(y)^{p}\omega^n(y)\right)^{\frac{1}{p}}
\end{align*}
so the claim follows by again combining expressions, and taking $\epsilon \searrow 0$. 
\end{proof}

Using an improved Kato inequality, we now show that for any holomorphic section $u$ of $L^m$ for suitable line bundles $L$, the quantities $|u|_{h}^{\alpha},|\nabla^{h}u|_{h}^{\alpha}$
satisfy the hypotheses of Lemma \ref{elliptic} for appropriate
$\alpha\in(0,1)$. This will be used in the proof of the $C^{0}$ and $C^{1}$ estimates
for the holomorphic sections of $L^m$ constructed
using the H\"ormander technique. 
\begin{lem}
\label{C1estimate} Suppose $n\geq 2$. For any precompact open set $B \subseteq X$, there exists $C=C(\lambda,n,B)<\infty$ such that the following holds. Suppose $(L,h)$ is a holomorphic Hermitian line bundle on $X^{\operatorname{reg}}$ with curvature $\Theta_h= m \omega$ for some $m\in \mathbb{N}$. For any $x_0 \in X$ and $r\in (0,\frac{1}{5}]$ with $B(x_0,50r)\subseteq B$, and any $u\in H^{0}(B(x_0,50r)\cap X^{\operatorname{reg}},L^{m})$
satisfying $\int_{B(x_0,50r)\cap X^{\operatorname{reg}}}|u|_{h}^{2}\omega^{n}<\infty$, we then have
\begin{equation} \label{goodC0estimate} \sup_{B(x_0,r)\cap X^{\operatorname{reg}}} |u|_h^2 \leq  \frac{Ce^{Cm r^2}}{r^{2n}} \int_{B(x_0,50r) \cap X^{\operatorname{reg}}} |u|_h^2 \omega^n. \end{equation}
\begin{equation} \label{goodC1estimate} \sup_{B(x_0,r)\cap X^{\operatorname{reg}}} |\nabla^h u|_{g\otimes h}^2 \leq  \left( m +\frac{1}{r^2}\right)\frac{Ce^{Cm r^2}}{r^{2n}} \int_{B(x_0,50r) \cap X^{\operatorname{reg}}} |u|_h^2 \omega^n. \end{equation}

\end{lem}

\begin{proof}
We compute
\begin{align*}
\Delta|u|_h^{2}= & g^{\overline{j}i}h\nabla_{i}\nabla_{\overline{j}}\left(u\overline{u}\right)\\
= & g^{\overline{j}i}h\nabla_{i}(u\overline{\nabla_{j}u})\\
= & g^{\overline{j}i}h\nabla_{i}u\overline{\nabla_{j}u}-g^{\overline{j}i}hu\overline{[\nabla_{j},\nabla_{\overline{i}}]u}\\
= & |\nabla^{h}u|_{g\otimes h}^{2}-nm |u|_h^{2},
\end{align*}
so that from 
\[
|\nabla|u|_h^{2}|_g^{2}=g^{\overline{j}i}(hu\overline{\nabla_{j}u}\cdot h \overline{u}\nabla_{i}u)=|u|_h^{2}|\nabla^{h}u|_{g\otimes h}^{2},
\]
it follows that for any $\alpha \in (0,1)$ and $\epsilon>0$, we have $v:=(|u|_h^{2}+\epsilon)^{\frac{\alpha}{2}}$ satisfies
\begin{align*}
\Delta v= & \frac{\alpha}{2}g^{\overline{j}i}\nabla_{i}\left(\frac{\nabla_{\overline{j}}|u|_h^{2}}{(|u|_h^{2}+\epsilon)^{1-\frac{\alpha}{2}}}\right)\\
= & \frac{\alpha}{2}\left(\frac{\Delta|u|_h^{2}}{(|u|_h^{2}+\epsilon)^{1-\frac{\alpha}{2}}}-\left(1-\frac{\alpha}{2}\right)\frac{|\nabla|u|_h^{2}|_g^{2}}{(|u|_h^{2}+\epsilon)^{2-\frac{\alpha}{2}}}\right)\\
= & \frac{\alpha}{2(|u|_h^{2}+\epsilon)^{1-\frac{\alpha}{2}}}\left(|\nabla^{h}u|_{g\otimes h}^{2}-nm|u|_h^{2}-\left(1-\frac{\alpha}{2}\right)\frac{|u|_h^{2}}{|u|_h^{2}+\epsilon}|\nabla^{h}u|_{g\otimes h}^{2}\right)\\
\geq & -\frac{nm\alpha|u|_h^{2}}{2(|u|_h^{2}+\epsilon)^{1-\frac{\alpha}{2}}}\\
\geq & -Cm v.
\end{align*}
Moreover, $|u|_h\in L^{2}(B(x_0,50r)\cap X^{\operatorname{reg}})$ implies $v:= (|u|_h^2+\epsilon)^{\frac{1}{4}}\in L^{4}(B(x_0,50r)\cap X^{\operatorname{reg}})$
(choosing $\alpha =\frac{1}{2}$). We can thus apply
Lemma \ref{elliptic}, replacing $r$ with $10r$ and taking $p=4$ in order to obtain
$$\sup_{B(x_0,10r)\cap X^{\operatorname{reg}}} (|u|_h^2+\epsilon)^{\frac{1}{4}} \leq C(\lambda,B)e^{Cmr^2} \left( \frac{1}{r^{2n}}\int_{B(x_0,50r)\cap X^{\operatorname{reg}}} (|u|_h^2+\epsilon) \omega^n \right)^{\frac{1}{4}}.$$
Taking $\epsilon \to 0$ gives
\begin{equation} \label{eq:C0} \sup_{B(x_0,10r)\cap X^{\operatorname{reg}}}|u|_h^2 \leq \frac{C(\lambda,B)e^{Cmr^2}}{r^{2n}} \int_{B(x_0,50r)\cap X^{\operatorname{reg}}}
|u|_h^2 \omega^n. \end{equation}

Next, we let $\phi_r$ be as in \eqref{phicutoff}, supported in $B(x_0,2r)$, and let $\eta_{\epsilon}$ be as in Lemma \ref{etacutoff}. Integrate $\Delta|u|_{h}^{2}=|\nabla^{h}u|_{g\otimes h}^{2}-nm|u|_{g\otimes h}^{2}$
against $\phi_{5r}^2\eta_{\epsilon}^2$ and use Cauchy's inequality to obtain
\begin{align*} \int_{X^{\operatorname{reg}}} |\nabla^h u|_{g\otimes h}^2 \phi_{5r}^2\eta_{\epsilon}^2\omega^n  =& -2\text{Re}\int_{X^{\operatorname{reg}}} \langle \nabla |u|_h^2,\overline{\nabla}(\phi_{5r}\eta_{\epsilon})^2\rangle \omega^n + nm \int_{B(x_0,10r)\cap X^{\operatorname{reg}}}|u|_h^2 \omega^n \\  \leq& \frac{1}{2}\int_{X^{\operatorname{reg}}} |\nabla^h u|_{g\otimes h}^2 \phi_{5r}^2 \eta_{\epsilon}^2\omega^n + C\int_{X^{\operatorname{reg}}} |u|_h^2 (\phi_{5r}^2|\nabla \eta_{\epsilon}|^2 + \eta_{\epsilon}^2|\nabla \phi_{5r}|^2)\omega^n \\ &+ nm \int_{B(x_0,10r)\cap X^{\operatorname{reg}}} |u|_h^2 \omega^n. 
\end{align*}
Because $\sup_{B(x_0,10r)\cap X^{\operatorname{reg}}} |u|_h<\infty$, we can take $\epsilon\searrow 0$ to obtain
$$\int_{B(x_0,5r)\cap X^{\operatorname{reg}}}|\nabla^h u|_{g\otimes h}^2 \omega^n \leq C\left( m + \frac{1}{r^2} \right) \int_{B(x_0,10r)} |u|_h^2 \omega^n.$$
To obtain the $C^1$ estimate for $u$, we use will use the identity
\begin{align*}
\Delta|\nabla^{h}u|_{g\otimes h}^{2}= & g^{\overline{j}i}g^{\overline{\ell}k}\nabla_{i}\nabla_{\overline{j}}\left(\nabla_{k}u\overline{\nabla_{\ell}u}\right)\\
= & g^{\overline{j}i}g^{\overline{\ell}k}\nabla_{i}\left(-[\nabla_{k},\nabla_{\overline{j}}]u\cdot\overline{\nabla_{\ell}u}+\nabla_{k}u\overline{\nabla_{j}\nabla_{\ell}u}\right)\\
= & -mg^{\overline{\ell}i}\nabla_{i}\left(u\overline{\nabla_{\ell}u}\right)+|\nabla^{h}\nabla^{h}u|_{g\otimes h}^{2}-g^{\overline{j}i}g^{\overline{\ell}k}\nabla_{k}u\overline{[\nabla_{j},\nabla_{\overline{i}}]\nabla_{\ell}u}-g^{\overline{j}i}g^{\overline{\ell}k}\nabla_{k}u\overline{\nabla_{j}[\nabla_{\ell},\nabla_{\overline{i}}]u}\\
= & -m|\nabla^{h}u|_{g\otimes h}^{2}+mg^{\overline{\ell}i}u\overline{[\nabla_{\ell},\nabla_{\overline{i}}]u}+|\nabla^{h}\nabla^{h}u|_{g\otimes h}^{2}+g^{\overline{j}i}g^{\overline{\ell}k}\nabla_{k}u\overline{R_{j\overline{i}\ell\overline{p}}g^{\overline{p}q}\nabla_{q}u}\\
 & -nm|\nabla^{h}u|_{g\otimes h}^{2}-m|\nabla^{h}u|_{g\otimes h}^{2}\\
= & |\nabla^{h}\nabla^{h}u|_{g\otimes h}^{2}-\left((n+2)m-\lambda\right)|\nabla^{h}u|_{g\otimes h}^{2}+nm^{2}|u|_{h}^{2},
\end{align*}
as well as the following refined Kato inequality:
\begin{align*}
|\nabla|\nabla^{h}u|_{g\otimes h}^{2}|_g^{2}= & h^2 g^{\overline{j}i}\nabla_{i}(g^{\overline{\ell}k}\nabla_{k}u\overline{\nabla_{\ell}u})\nabla_{\overline{j}}(g^{\overline{q}p}\nabla_{p}u\overline{\nabla_{q}u})\\
= & h^2 g^{\overline{j}i}g^{\overline{\ell}k}g^{\overline{q}p}\left(\nabla_{i}\nabla_{k}u\overline{\nabla_{\ell}u}-\nabla_{k}u\overline{[\nabla_{\ell},\nabla_{\overline{i}}]u}\right)\left(\nabla_{p}u\overline{\nabla_{j}\nabla_{q}u}-\overline{\nabla_{q}u}[\nabla_{p},\nabla_{\overline{j}}]u\right)\\
= & h^2 g^{\overline{j}i}g^{\overline{\ell}k}g^{\overline{q}p}\left(\nabla_{i}\nabla_{k}u\overline{\nabla_{\ell}u}- m g_{i\overline{\ell}}\bar{u}\nabla_{k}u\right)\left(\nabla_{p}u\overline{\nabla_{j}\nabla_{q}u}-m g_{p\overline{j}}u\overline{\nabla_{q}u}\right)\\
\leq & h^2 g^{\overline{j}i}g^{\overline{\ell}k}g^{\overline{q}p}\nabla_{i}\nabla_{k}u\overline{\nabla_{j}\nabla_{q}u}\nabla_{p}u\overline{\nabla_{\ell}u}+|\nabla^{h}u|_{g\otimes h}^{2}(2|\nabla^{h}\nabla^{h}u|_{g\otimes h}\cdot m|u|_h+m^2|u|_h^{2})\\
\leq & |\nabla^{h}\nabla^{h}u|_{g\otimes h}^{2}|\nabla^{h}u|_{g\otimes h}^{2}+|\nabla^{h}u|_{g\otimes h}^{2}(2m|u|_h |\nabla^{h}\nabla^{h}u|_{g\otimes h}+m^2|u|_h^{2}),
\end{align*}
where in the last line, we used the Cauchy-Schwarz inequality (for ease of computation, one may assume $g_{i\overline{j}}=\delta_{ij}$ and $\nabla_i u = |\nabla^h u|_{g\otimes h}\delta_{i1}$ at a given point). 

For any $\alpha \in (0,1)$ and $\epsilon>0$, we combine the above expressions to obtain
\begin{align*} 
\Delta &\left(|\nabla^{h}u|_{g\otimes h}^{2}+\epsilon\right)^{\frac{\alpha}{2}} \\
= & \frac{\alpha}{2}g^{\overline{j}i}\nabla_{i}\left(\frac{\nabla_{\overline{j}}|\nabla^{h}u|_{g\otimes h}^{2}}{(|\nabla^{h}u|_{g\otimes h}^{2}+\epsilon)^{1-\frac{\alpha}{2}}}\right)\\
= & \frac{\alpha}{2}\left(\frac{\Delta|\nabla^{h}u|_{g\otimes h}^{2\alpha}}{(|\nabla^{h}u|_{g\otimes h}^{2}+\epsilon)^{1-\frac{\alpha}{2}}}-\left(1-\frac{\alpha}{2}\right)\frac{|\nabla|\nabla^{h}u|_{g\otimes h}^{2}|_g^2}{(|\nabla^{h}u|_{g\otimes h}^{2}+\epsilon)^{2-\frac{\alpha}{2}}}\right)\\
\geq & \frac{\alpha}{2}\left(\frac{|\nabla^{h}\nabla^{h}u|_{g\otimes h}^{2}-Cm|\nabla^{h}u|_{g\otimes h}^{2}+n m^2|u|_h^2}{(|\nabla^{h}u|_{g\otimes h}^{2}+\epsilon)^{1-\frac{\alpha}{2}}} \right) \\&- \frac{\alpha}{2}\left(1-\frac{\alpha}{2}\right)\frac{|\nabla^{h}\nabla^{h}u|_{g\otimes h}^{2}|\nabla^{h}u|_{g\otimes h}^{2}+|\nabla^{h}u|_{g\otimes h}^{2}(2m|u|_h |\nabla^{h}\nabla^{h}u|_{g\otimes h}+m^2 |u|_h^2)}{(|\nabla^{h}u|_{g\otimes h}^{2}+\epsilon)^{2-\frac{\alpha}{2}}}
\\
\geq & -Cm(|\nabla^h u|_{g\otimes h}^2+\epsilon)^{\frac{\alpha}{2}} + \frac{\alpha m^2}{2} \left( n+1-\frac{2}{\alpha}\right)\frac{|u|_h^2 }{(|\nabla^h u|_{g\otimes h}^2+\epsilon)^{1-\frac{\alpha}{2}}}
\end{align*}
Because $n\geq 2$, have $\alpha :=\frac{2}{n+1}\in (0,1)$. It follows that $v:=(|\nabla^{h}u|_{g\otimes h}^{2}+\epsilon)^{\frac{\alpha}{2}}$
satisfies $\Delta v\geq-Cmv$. Moreover, we know from $|\nabla^{h}u|_{g\otimes h}\in L^{2}(B(x_0,5r)\cap X^{\operatorname{reg}})$
that $v\in L_{\text{loc}}^{p}(B(x_0,5r)\cap X^{\operatorname{reg}})$, where $p:=n+1>2$. We can thus apply Lemma \ref{elliptic} and then take $\epsilon \searrow 0$ to get the remaining claim. 
\end{proof}

\section{Construction of peaked holomorphic Sections}
Our goal in this section is to prove Theorems~\ref{QCartierTheorem} and \ref{KETheorem}. 
Throughout this section, we suppose that $(X,\omega)$ is a rough K\"ahler-Einstein variety.

\begin{prop}
\label{perturb} For any $\epsilon>0$, there exist $\zeta=\zeta(\epsilon)>0$, and $D_{0}=D_{0}(\epsilon)<\infty$ such that the following holds
for any $\ell\in\mathbb{N}^{\times}$. Suppose there exist $x_0 \in X$, $D\geq D_0$, a Stein neighborhood $B\subseteq X$ of $x_0$ such that $B_{g_{\ell}}(x_0,100D) \cap X^{\operatorname{reg}}\subseteq B$, and a holomorphic Hermitian line bundle $(L,h)$ on $B \cap X^{\operatorname{reg}}$ such that $\Theta_h = \omega$. Set $\omega_{\ell}:=\ell \omega$, so that $h_{\ell}:=h^{\otimes \ell}$ is a Hermitian metric on $L^{\ell}$ satisfying $\Theta_{h_{\ell}}=\omega_{\ell}$. Assume $v\in C_{c}^{\infty}(B\cap X^{\operatorname{reg}},L^{\ell})$,  $U\subseteq X^{\operatorname{reg}}$ is open, and that the following hold
for some $\ell \in \mathbb{N}^{\times}$:

$(i)$ $\int_{B\cap X^{\operatorname{reg}}}|v|_{h_{\ell}}^{2}\omega_{\ell}^{n}<(1+\zeta)(2\pi)^{n},$ 

$(ii)$ $\sup_{B\cap U}\left|e^{-\frac{1}{2}d_{g_{\ell}}^{2}(x_0,\cdot)}-|v|_{h_{\ell}}^{2}\right|<\zeta$, 

$(iii)$ $\int_{B\cap X^{\operatorname{reg}}}|\overline{\partial}v|_{\omega_{\ell}\otimes h_{\ell}}^{2}\omega_{\ell}^{n}<\zeta$,

$(iv)$ $\sup_{B\cap U}|\overline{\partial}v|_{\omega_{\ell}\otimes h_{\ell}}^{2}<\zeta$, 

$(v)$ for any $z\in B_{g_{\ell}}(x_0,100D)$ with $B_{g_{\ell}}(z,\zeta)\subseteq X^{\operatorname{reg}}$ and $\sup_{B_{g_{\ell}}(z,\zeta)}|Rm|_{g_{\ell}}\leq\zeta^{-2}$, we have $z\in U$,

$(vi)$ $\operatorname{supp}(v)\subseteq B_{g_{\ell}}(x_0,50D)$,

\noindent Then there exists a holomorphic section $u\in H^{0}(B_{g_{\ell}}(x_0,D)\cap X^{\operatorname{reg}},L^{\ell})$
satisfying the following:

$(a)$ $\int_{B\cap X^{\operatorname{reg}}}|u|_{h_{\ell}}^{2}\omega_{\ell}^n<(1+\epsilon)(2\pi)^n$,

$(b)$ $\sup_{B_{g_{\ell}}(x_0,1)\cap X^{\operatorname{reg}}}\left|e^{-\frac{1}{2}d_{g_{\ell}}^{2}(x_0,\cdot)}-|u|_{h_{\ell}}^{2}\right|<\epsilon.$
\end{prop}

\begin{proof}
Because
$B\setminus X^{\operatorname{reg}}$ is a complex analytic subset of the Stein space $B$, we can apply \cite[Theorem 0.2]{demaillyEstimationsL2Pour1982}
to conclude that $B\cap X^{\operatorname{reg}}$ admits a complete K\"ahler metric.
Using 
$$L^{\ell} \cong (L^{\ell}\otimes K_{X^{\operatorname{reg}}}^{-1})\otimes K_{X^{\operatorname{reg}}},$$
we can identify $v$ with some $\widetilde{v}\in\mathcal{A}_{c}^{n,0}(B\cap X^{\operatorname{reg}},L^{\ell}\otimes K_{X^{\operatorname{reg}}}^{-1})$. Let $\varphi$ be a plurisubharmonic function on $B$ such that $\sqrt{-1}\partial \overline{\partial} \varphi=\omega$. Set $\widetilde{h}:=e^{\lambda \varphi} h_{\ell}\otimes\omega_{\ell}^n$,
which is a Hermitian metric on $L^{\ell}\otimes K_{X^{\operatorname{reg}}}^{-1}$ satisfying
$|\widetilde{v}|_{\widetilde{h}}^{2}=|v|_{h_{\ell}}^{2}e^{\lambda \varphi}\omega_{\ell}^{n}$
and 
\[
\Theta_{\widetilde{h}}=Rc(\omega)-\lambda\sqrt{-1}\partial\overline{\partial}\varphi+\omega_{\ell}=\omega_{\ell},
\]
we can apply \cite[Theorem 8.6.1]{demaillyjean-pierreComplexAnalyticDifferential}
with K\"ahler metric $\omega_{\ell}$ and holomorphic Hermitian
line bundle $(K_{X^{\operatorname{reg}}}^{-1}\otimes L^{\ell},\widetilde{h})$ to obtain
an $L^{2}$ $(n,0)$-form $\widetilde{w}$ on $B$ valued in $K_{X^{\operatorname{reg}}}^{-1}\otimes L^{\ell}$
such that $\overline{\partial}\widetilde{w}=\overline{\partial}\widetilde{v}$
and 
\[
\int_{B\cap X^{\operatorname{reg}}}|w|_{h_{\ell}}^{2}\omega_{\ell}^{n}\leq C\int_{B\cap X^{\operatorname{reg}}}|w|_{\widetilde{h}}^{2}\leq C\int_{B\cap X^{\operatorname{reg}}}|\overline{\partial}\widetilde{v}|_{\omega_{\ell}\otimes\widetilde{h}}^{2}\leq C\int_{B\cap X^{\operatorname{reg}}}|\overline{\partial}v|_{\omega_{\ell}\otimes h_{\ell}}^{2}\omega_{\ell}^{n}<C\zeta
\]
where $w$ is the $L^{2}$ section of $L^{\ell}$
corresponding to $\widetilde{w}$, and we used $(iii)$. Set $u:=v-w\in H^{0}(B\cap X^{\operatorname{reg}},L^{\ell})$,
which satisfies
\[
\int_{B\cap X^{\operatorname{reg}}}|u|_{h_{\ell}}^{2}\omega_{\ell}^{n}\leq (1+C\zeta)(2\pi)^n
\]
by $(i)$. Let $\epsilon'>0$. By Lemma \ref{improvedepsreg}(i) and the estimate for the $(2n-1)$th quantitative stratum \cite{RCDquantstrata}, the set $\Sigma(\epsilon')$ of points $z\in B_{g_{\ell}}(x_0,100)$ which do not satisfy $B_{g_{\ell}}(z,\epsilon')\subseteq X^{\operatorname{reg}}$ and $\sup_{B_{g_{\ell}}(z,\epsilon')} |Rm|_{g_{\ell}} \leq (\epsilon')^{-2}$ has $g_{\ell}$-volume at most $C(B,\lambda)(\epsilon')^{\frac{1}{2}}$. In particular, 
$$\text{vol}_{g_{\ell}}(\Sigma(\epsilon')\cap B_{g_{\ell}}(x_0,100))\leq \text{vol}_{g_{\ell}}(B(z,(\epsilon')^{\frac{1}{4n}})),$$
so for any $z\in B_{g_{\ell}}(x_0,1)$, there exists $z'\in B_{g_{\ell}}(z,C(\epsilon')^{\frac{1}{4n}})\setminus \Sigma(\epsilon')$. By assumptions $(iv),(v)$, and by $\int_{B\cap X^{\operatorname{reg}}}|w|_{h_{\ell}}^2\omega_{\ell}^n \leq C\zeta$ and local elliptic regularity near $z'$, we have $|w|_{h_{\ell}}\leq \frac{\epsilon}{4}$ if $\zeta=\zeta(\epsilon',\epsilon)$ is sufficiently small. Combining with assumption (ii) yields
$$|e^{-\frac{1}{2}d_{g_{\ell}}(x_0,z')}-|u|_{h_{\ell}}^2(z')| <\frac{\epsilon}{2}$$
if $\zeta = \zeta(\epsilon',\epsilon)$ is sufficiently small. By applying \eqref{goodC1estimate} with $r=2\ell^{-2}$, we obtain
$$\sup_{B_{g_{\ell}}(x_0,2)}|\nabla^h u|_{g_{\ell}\otimes h_{\ell}} \leq C(\lambda,B).$$
Thus $|e^{-\frac{1}{2}d_{g_{\ell}}(x_0,\cdot)}-|v|_{h_{\ell}}^2|$ is $C(\lambda,B)$-Lipschitz with respect to $g_{\ell}$, yielding 
$$|e^{-\frac{1}{2}d_{g_{\ell}}(x_0,z)}-|u|_{h_{\ell}}^2(z)|\leq C(\lambda,B)(\epsilon')^{\frac{1}{4n}}+\frac{\epsilon}{2}.$$
Because $z\in B_{g_{\ell}}(x_0,1)$ was arbitrary, the remaining claim follows by choosing $\epsilon'=\epsilon'(\epsilon)>0$ small and then $\zeta=\zeta(\epsilon,\epsilon')>0$ small. 
\end{proof}

\begin{proof}[Proof of Theorem \ref{QCartierTheorem}] Let $B\subseteq X$ be a Stein neighborhood of a given point $x_0 \in X$ such that $\omega = \sqrt{-1}\partial \overline{\partial}\varphi$ for some plurisubharmonic function $\varphi$ on $B$. Then $(L,h):=(K_{X^{\operatorname{reg}}}^{-1}|_B,e^{(\lambda-1) \varphi}\omega^n)$ is a holomorphic Hermitian line bundle on $B\cap X^{\operatorname{reg}}$ with curvature $\omega$. Fix $\epsilon>0$, and let $\zeta=\zeta(\epsilon)>0$ and $D_{0}=D_{0}(\epsilon)<\infty$
be as in Proposition \ref{perturb}. 
By Lemma \ref{improvedepsreg}(ii), the singular set of any iterated tangent cone of $B$ is closed; because the singular set also has Hausdorff codimension at least 4 by Proposition \ref{codim4}, we can argue as in \cite[Section 3.2.2]{donaldsonGromovHausdorffLimitsKahler2014} to obtain
some $\ell\in\mathbb{N}^{\times}$ and $v\in C_{c}^{\infty}(B\cap X^{\operatorname{reg}},L^{\ell})$
satisfying hypotheses $(i)-(vi)$ of Proposition \ref{perturb}. If
we choose $\epsilon>0$ sufficiently small, then there exists $C<\infty$ such that the section $u\in H^{0}(B\cap X^{\operatorname{reg}},L^{\ell})$
guaranteed by Proposition \ref{perturb} then satisfies 
$C^{-1}\leq |u|_{h_{\ell}} \leq C$
on $B(x_0,1)\cap X^{\operatorname{reg}}$. In particular, $L^{\ell}$ extends to a line bundle over all of $X$, so that $L$ is $\mathbb{Q}$-Cartier.
\end{proof}

\begin{proof}[Proof of Theorem \ref{KETheorem}] 
\noindent (i) Fix $x_0 \in X$, and let $B$ be the intersection of $X$ with a Euclidean ball centered at $x_0$ with respect to a local holomorphic embedding of $X$ near $x_0$. By \ref{condition:bddpotentials}, we can assume that there exists plurisubharmonic $\varphi \in L^{\infty}(B)$ satisfying $\omega = \sqrt{-1}\partial \overline{\partial}\varphi$ on $B$.

Set $L:=K_{X^{\operatorname{reg}}}^{-1}$ and $h:=e^{(\lambda-1)\varphi}\omega^n$. For $\ell \in \mathbb{N}$ large, consider the section $u\in H^0(B\cap X^{\operatorname{reg}},L^{\ell})$ constructed in the proof of Theorem \ref{QCartierTheorem}, so that
$$C^{-1}\leq |u|_{h_{\ell}} \leq C$$
on $B_{g_{\ell}}(x_0,1)$. Let $u^{\ast} \in H^0 (B_{g_{\ell}}(x_0,1),L^{-\ell})$ be the dual section. Because $(X,\omega)$ has finite volume on bounded sets, it follows that 
$$\Omega := \left( (\sqrt{-1})^{n^2\ell}u^{\ast}\wedge \overline{u^{\ast}} \right)^{\frac{1}{\ell}} \in \mathcal{A}^{n,n}(B_{g_{\ell}}(x_0,1)\cap X^{\operatorname{reg}})$$
is an adapted volume form satisfying
\[
\int_{B_{g_{\ell}}(x_0,1)\cap X^{\operatorname{reg}}}\Omega=\int_{B_{g_{\ell}}(x_0,1)\cap X^{\operatorname{reg}}}|u|_{h_{\ell}}^{-\frac{2}{\ell}}e^{(\lambda-1)\varphi}\omega^{n}\leq C\int_{B_{g_{\ell}}(x_0,1)\cap X^{\operatorname{reg}}}\omega^{n}<\infty.
\]
By \cite[Lemma 6.4]{EGZ},
$X$ has log-terminal singularities.

(ii) For any $\ell \in \mathbb{N}$, 
$(K_{X^{\operatorname{reg}}}^{-\ell}|_{B\cap X^{\operatorname{reg}}},(e^{\lambda \varphi}\omega^{n})^{\otimes\ell})$
is a flat holomorphic line bundle on $B\cap X^{\operatorname{reg}}$. Because $K_{X^{\operatorname{reg}}}^{-\ell}$ is trivial in a neighborhood of $x_0$ for some $\ell \in \mathbb{N}^{\times}$, we can choose $B$ so that the holonomy of $(K_{X^{\operatorname{reg}}}^{-\ell}|_{B\cap X^{\operatorname{reg}}},(e^{\lambda \varphi}\omega^{n})^{\otimes\ell})$ is trivial for $\ell \in \mathbb{N}$ sufficiently large. Thus $K_{X^{\operatorname{reg}}}^{-\ell}|_{B\cap X^{\operatorname{reg}}}$
admits a parallel section $\sigma\in H^{0}(B\cap X^{\operatorname{reg}},K_{X^{\operatorname{reg}}}^{-\ell})$
for some $\ell\in\mathbb{N}$, with respect to the Hermitian metric $h$
on $K_{X^{\operatorname{reg}}}^{-\ell}$ corresponding to $(e^{\lambda \varphi}\omega^{n})^{\otimes\ell}.$
Set 
\[
v:=\left((\sqrt{-1})^{n^{2}\ell}\sigma\wedge\overline{\sigma}\right)^{-\frac{1}{\ell}}\in\mathcal{A}^{n,n}(B\cap X^{\operatorname{reg}}),
\]
so that on $B\cap X^{\operatorname{reg}}$, 
\[
\log\frac{v}{e^{\lambda \varphi}\omega^{n}}=-\frac{1}{\ell}\log|\sigma|_{h}^{2}
\]
is constant. In other words,
\[
\omega^{n}= ce^{-\lambda \varphi} v
\]
on $B\cap X^{\operatorname{reg}}$ for some $c\in(0,\infty)$. Let $j:X^{\operatorname{reg}}\hookrightarrow X$ be the inclusion map. Because $\omega$
has bounded K\"ahler potential $\varphi$ on $B$, the complex Monge-Ampere measure of $\omega$ on $B$ is well-defined, and equal to $j_\ast \omega^n$. Because $\int_{B\cap X^{\operatorname{reg}}}v<\infty$, we also know $j_{\ast}v$ is a well-defined Radon measure on $X$, and the above equality holds in the sense of Radon measures on all of $B$. 

(iii) Given (ii), this is proved in \cite[Appendix]{donaldsonGromovHausdorffLimits2017}.
\end{proof}

\section{Applications to limits of K\"ahler-Einstein Metrics and K\"ahler-Ricci flow}

In this section, we assume that $(X,d)$ is a metric cone with vertex $o\in X$.
Moreover, we assume $X$ arises as a limit in one of the following
two settings, for some $Y<\infty$:
\begin{lyxlist}{00.00.0000}
\item [{\textbf{$(A)$}}] $(M_{i}^{n},J_{i},g_{i},x_{i})$ is a sequence
of complete K\"ahler manifolds of complex dimension $n$ satisfying
$|Rc|_{g_{i}}\leq1$ and $\text{Vol}_{g_{i}}(B(x_{i},1))\geq Y^{-1}$
which converge in the pointed Gromov-Hausdorff sense to $(X,d,o)$.

\item [{\textbf{$(B)$}}] $(M_{i}^{n},J_{i},(g_{i,t})_{t\in[-T_{i},0]})$
is a sequence of compact K\"ahler-Ricci flows of complex dimension $n$, and $x_{i}\in M_{i}$
are points satisfying $\mathcal{N}_{x_{i},0}(1)\geq-Y$, such that (see \cite[Section 5.1]{bamlerCompactnessTheorySpace2023} for definitions)
\[
(M_{i},(g_{i,t})_{t\in[-T_{i},0]},(\nu_{x_{i},0;t})_{t\in[-T_{i},0]})\xrightarrow[i\to\infty]{\mathbb{F}}(\mathcal{X},(\mu_{t})_{t\in(-\infty,0]}),
\]
where $\mathcal{X}$ is a static cone modeled on $(X,d)$ in the sense of \cite[Definition 3.60]{bamlerCompactnessTheorySpace2023}, and $$d\mu_{t}=(2\pi|t|)^{-n}e^{-\frac{d^{2}(o,\cdot)}{2|t|}}\frac{1}{n!}\omega_{t}^n.$$ 
\end{lyxlist}

Given $(A)$, we let $\mathcal{R} \subseteq X$ denote the metric regular set of $X$ in the sense of \cite{cheegerStructureSpacesRicci1997}, so that there is a Ricci-flat K\"ahler cone structure $(J,g)$ on $\mathcal{R}$ such that the metric completion of $(\mathcal{R},d_g)$ is $(X,d)$. It was shown in \cite{liuGromovHausdorffLimitsAhler2019} that $X$ naturally has the
structure of a normal affine algebraic variety with $X^{\operatorname{reg}}=\mathcal{R}$ as complex manifolds, whose structure sheaf $\mathcal{O}_{X}$ consists of holomorphic functions
on $X^{\operatorname{reg}}$ which are bounded on bounded subsets of $X^{\operatorname{reg}}$. 

Given $(B)$, we let $\mathcal{R}_X \subseteq X$ denote the regular set of $X$ as in \cite[Definition 2.15]{bamlerStructureTheoryNoncollapsed2021}, so that $\mathcal{R}_X$ possesses a K\"ahler cone structure $(J,g)$ such that the metric completion of $(\mathcal{R}_X,d_g)$ is $(X,d)$ by \cite[Theorem 2.18]{bamlerStructureTheoryNoncollapsed2021}. It was shown in \cite{hallgrenKahlerRicciTangentFlows2023} that $X$ admits the structure of a normal affine algebraic variety $X$ as above, with $\mathcal{R}_X=X^{\operatorname{reg}}$. 

Given either $(A)$ or $(B)$, there are holomorphic embeddings $X\hookrightarrow\mathbb{C}^{N}$
(where $N$ can be taken to be the dimension of the Zariski tangent
cone at $o$) such that the real-holomorphic vector field $J\nabla(\frac{d^{2}(o,\cdot)}{2})$
can be identified with the restriction of the holomorphic vector field
$\xi=\sum_{\alpha=1}^{N}\sqrt{-1}w_{\alpha}z_{\alpha}\frac{\partial}{\partial z_{\alpha}}$
on $\mathbb{C}^{N}$, where $w_{\alpha}>0$ for $1\leq\alpha\leq N$.
Thus $(X,\xi)$ is naturally a polarized affine variety.

Suppose assumption $(B)$ holds, and write $r:=d(o,\cdot)$. In the proof of the next lemma, we use the following notational convention: we let $\Psi(a|b_1,...,b_{\ell})$ denote a quantity depending on parameters $a,b_1,...,b_{\ell}$, which satisfies
$$\lim_{a\to0} \Psi(a|b_1,...,b_{\ell})=0$$
for any fixed $b_1,...,b_{\ell}$.
\begin{lem}
\label{LipHarmonic} Suppose $u\in L_{\text{loc}}^{2}(X)$ satisfies
$|\nabla u|\in L_{\text{loc}}^{2}(X)$, $\Delta u=0$ on $X^{\operatorname{reg}}$ and $\mathcal{L}_{\nabla r}u=mu$
for some $m\in\mathbb{N}$. Then $u$ extends to a locally Lipschitz
function on $X$. 
\end{lem}

\begin{proof}
This holds even without assuming the Ricci flows $(M_i,(g_{i,t})_{t\in [-T_i,0]})$ are K\"ahler, with no added difficulty. We prove it in this generality, letting $dg$ denote the Riemannian volume measure on the regular part $\mathcal{R}$ of the cone, and letting $\nabla$ denote the Levi-Civita connection, instead of just its $(1,0)$-part. By \cite[Proof of Lemma 5.1]{hallgrenKahlerRicciTangentFlows2023}, for any $\epsilon,r>0$, there exist locally Lipschitz $\eta_{\epsilon},\phi_r :X\to [0,1]$ satisfying $\mbox{supp}(\phi_r)\subseteq B(o,2r)$, $\phi_r|_{B(o,r)} \equiv 1$,
$r|\nabla \phi|+r^2|\Delta \phi|\leq C$, $\mbox{supp}(\eta_{\epsilon})\cap B(o,r) \subset \subset \mathcal{R}$, 
$$\lim_{\epsilon \searrow 0} \int_{\mathcal{R}\cap B(o,r)} |\nabla \eta_{\epsilon}|^{\frac{7}{2}} dg=0,$$
and such that for any compact subset $\mathcal{K}\subseteq \mathcal{R}$, we have $\eta_{\epsilon}|_{\mathcal{K}}\equiv1$ for sufficiently small $\epsilon =\epsilon(\mathcal{K})>0$. By \cite[Appendix]{hallgrenKahlerRicciTangentFlows2023}, there is a function $K$ satisfying the conclusions of Lemma \ref{heatkernel}, such that the constant $C(\mathcal{K})$ appearing in assertions (v),(vi) of this lemma can be taken independent of the compact set $\mathcal{K}$. Because $\operatorname{supp}(\eta_{\epsilon}\phi_{r}u)\in C_{c}^{\infty}(\mathcal{R})$,
we can integrate by parts to obtain
\begin{align*}
\frac{d}{dt}\int_{X}u(y)\eta_{\epsilon}(y)\phi_{r}(y)K(x,y,1-t)dg(y) \hspace{-50 mm} &\\
=&-2\int_{X}\langle\nabla u,\eta_{\epsilon}\nabla\phi_{r}+\phi_{r}\nabla\eta_{\epsilon}\rangle(y)K(x,y,1-t)dg(y)\\
 & -\int_{X}u(y)\left(2\langle\nabla\eta_{\epsilon}(y),\nabla\phi_{r}(y)\rangle+\eta_{\epsilon}(y)\Delta\phi_{r}(y)\right)K(x,y,1-t)dg(y)\\
 & -\int_{X}u(y)\phi_{r}(y)\Delta\eta_{\epsilon}(y)K(x,y,1-t)dg(y).
\end{align*}
More integration by parts gives
\begin{align*}
-\int_{X}u(y)\phi_{r}(y)&\Delta\eta_{\epsilon}(y)K(x,y,r^2-t)dg(y) \\=&\int_{X}\langle\nabla\eta_{\epsilon},u\nabla\phi_{r}+\phi_{r}\nabla u\rangle(y)K(x,y,1-t)dg(y)\\
&+ \int_X u(y)\phi_{r}(y)\langle \nabla \eta_{\epsilon},\nabla K(x,\cdot,r^2-t)\rangle(y)dg(y)
\end{align*}
and also
\begin{align*}
-\int_{X}\langle\nabla u,\eta_{\epsilon}\nabla\phi_{r}\rangle(y)&K(x,y,1-t)dg(y)
\\=&\int_{X}u(y)\left(\eta_{\epsilon}\Delta\phi_{r}+\langle\nabla\eta_{\epsilon},\nabla\phi_{r}\rangle\right)(y)K(x,y,1-t)dg(y)\\ & + \int_{X} u(y) \eta_{\epsilon}(y)\langle\nabla\phi_{r},\nabla K(x,\cdot,1-t)\rangle(y) dg(y)
\end{align*}
Now assume $x\in B(o,\frac{r}{2})\cap \mathcal{R}$, so that
there exists $r_0=r_0(x)>0$ such that $d(x,\operatorname{supp}(1-\eta_{\epsilon}))\geq r_0$
for all $\epsilon>0$ sufficiently small. Then we can estimate
\begin{align*}
\left|\int_{X}\langle\nabla u,\phi_{r}\nabla\eta_{\epsilon}\rangle(y)K(x,y,1-t)dg(y)\right| \hspace{-40 mm} &
\\ \leq & C\frac{1}{(1-t)^{n}}\exp\left(-\frac{r_{0}^{2}}{C(1-t)}\right)\left(\int_{B(o,2r)\cap \mathcal{R}}|\nabla u|^{2}dg\right)^{\frac{1}{2}}\left(\int_{B(o,2r)\cap \mathcal{R}}|\nabla\eta_{\epsilon}|^{2}dg\right)^{\frac{1}{2}}\\
\leq & C(r_{0},r)\left(\int_{B(o,2r)\cap \mathcal{R}}|\nabla u|^{2}dg\right)^{\frac{1}{2}}\left(\int_{B(o,2r)\cap \mathcal{R}}|\nabla\eta_{\epsilon}|^{2}(y)dg(y)\right)^{\frac{1}{2}}\\
\leq & \Psi(\epsilon|r_{0},r),
\end{align*}
and similarly
\[
\left|\int_{X}u(y)\langle\nabla\eta_{\epsilon},\nabla\phi_{r}\rangle(y)K(x,y,1-t)dg(y)\right|\leq\Psi(\epsilon|r_{0},r),
\]
\[
\left|\int_{X}\langle\nabla\eta_{\epsilon},u\phi_{r}\nabla K(x,\cdot,1-t)\rangle(y)dg(y)\right|\leq\Psi(\epsilon|r_{0},r),
\]
For $r\geq1$, we use $\mathcal{L}_{\nabla r}u=mu$ to estimate
\begin{align*}
\left|\int_{X}u(y)\eta_{\epsilon}(y)\Delta\phi_{r}(y)K(x,y,1-t)dg(y)\right|\leq & \frac{C}{r^{2}}\frac{r^{2n}}{(1-t)^{n}}\exp\left(-\frac{r^{2}}{C(1-t)}\right)\frac{1}{r^{2n}}\int_{B(o,2r)\cap \mathcal{R}}|u|dg\\
\leq & Cr^{m-2}\exp\left(-\frac{r^{2}}{C}\right)\\
= & \Psi(r^{-1}).
\end{align*}
Similarly, we have 
\[
\left|\int_{X}u(y)\eta_{\epsilon}(y)\langle\nabla\phi_{r}(y),\nabla K(x,y,1-t)\rangle dg(y)\right|\leq\Psi(r^{-1}).
\]
Combining expressions, we have
\[
\left|\frac{d}{dt}\int_{X}u(y)\eta_{\epsilon}(y)\phi_{r}(y)K(x,y,1-t)dg(y)\right|\leq\Psi(r^{-1})+\Psi(\epsilon|r_{0},r),
\]
so integrating from $t=0$ to $t=1$ gives
\[
\left|u(x)-\int_{X}u(y)\eta_{\epsilon}(y)\phi_{r}(y)K(x,y,1)dg(y)\right|\leq\Psi(r^{-1})+\Psi(\epsilon|r_{0},r).
\]
Because $uK(x,\cdot,1)\in L^{1}$, we can take $\epsilon\searrow0$
and then $r\to\infty$, appealing to the dominated convergence theorem
to obtain
\[
u(x)=\int_{X}u(y)K(x,y,1)dg(y).
\]
Because $r\mapsto\int_{B(o,r)}|u|dg(y)$ has polynomial growth, we
can use the Gaussian estimates for $K$ and $|\nabla K|$ to conclude
that $|\nabla u|$ is locally bounded. 
\end{proof}

\begin{prop} \label{coneisRCD}
If $X=C(Z)$ is a cone satisfying assumption $(B)$, then $(C(Z),d,\mathcal{H}^{2n})$
satisfies the $RCD(0,2n)$ condition.
\end{prop}

\begin{proof}
By \cite[Theorem 1.2]{Ketterer}, this is equivalent to showing that $Z$ satisfies the
$RCD(2n-1,2n)$ condition. We appeal to Honda's characterization \cite[Corollary 3.10]{hondaBakryEmeryConditions2018}, using the fact that $C(Z)^{\operatorname{reg}}$ is
Ricci-flat, and that $C(Z)$ satisfies the Sobolev to Lipschitz property.
Moreover, $C(Z)$ satisfies a Sobolev inequality by  \cite[Corollary 1.9]{chan2024noncollapsedmathbbflimitmetricsolitons}, so that because $\overline{B}(o,2)\setminus\overline{B}(o,\frac{1}{2})$
is quasi-isometric to the metric product $Z\times[\frac{1}{2},2]$,
$Z$ also satisfies a Sobolev inequality, hence also satisfies the
$L^{2}$-strong compactness condition. It remains to show that, given
any $v\in W^{1,2}(Z)$ satisfying $\Delta_{Z}v=-\lambda v$ for some
$\lambda\in[0,\infty)$, $v$ is locally Lipschitz. Because the function
$u:C(Z)\to\mathbb{R}$ defined by $u(r,y):=r^{\alpha}v(y)$ satisfies
\begin{align*}
\Delta u= & \left(\frac{\partial^{2}}{\partial r^{2}}+\frac{2n-1}{r}\frac{\partial}{\partial r}+\frac{1}{r^{2}}\Delta_{Z}\right)(r^{\alpha}v)\\
= & \left(\alpha(\alpha-1)+(2n-1)\alpha-\lambda\right)(r^{\alpha-2}v),
\end{align*}
if we choose
\[
\alpha:=\frac{1}{2}\left(-(2n-2)+\sqrt{(2n-2)^{2}+4\lambda}\right)>0,
\]
it follows that $u\in W_{\text{loc}}^{1,2}(C(Z))$ and $\Delta u=0$, $\mathcal{L}_{\nabla r}u=\alpha u$ on $C(Z)^{\operatorname{reg}}$. We can therefore apply Lemma \ref{LipHarmonic} to conclude that $u$
is locally Lipschitz, hence $v$ is Lipschitz. 
\end{proof}

We now restate a more precise version of Theorem \ref{applicationtoRF}. 

\begin{thm} If $(X,d)$ satisfies one of assumptions $(A),(B)$, then $X$ is a rough K\"ahler-Einstein variety.
\end{thm}
\begin{proof} Clearly property \ref{condition:einstein} of Definition \ref{rough} holds in either case. Given either (A) or (B), $X$ admits a holomorphic embedding $F:X\to \mathbb{C}^N$ by locally Lipschitz functions, so that $\text{tr}_{\omega}(F^{\ast} \omega_{\mathbb{C}^N}) = |dF|_{\omega,\omega_{\mathbb{C}^N}}^2$ implies Definition \ref{rough} \ref{condition:domination}. Because $\frac{1}{2}r^2$ is a locally bounded K\"ahler potential for $\omega$, it follows that $X$ satisfies property \ref{condition:bddpotentials} of Definition \ref{rough}. Given assumption (A), property \ref{condition:metricompletion} of Definition \ref{rough} follows from the fact that the RCD condition is stable under pointed Gromov-Hausdorff limits \cite[Theorem 2.7]{RCDstability}. Given assumption (B), this instead follows from Proposition \ref{coneisRCD}. 

Given assumption (A), property \ref{condition:epsregularity} follows from \cite[Theorem 3.2 and Remark 3.3]{andersoneps}. It remains to prove that property \ref{condition:epsregularity} holds under assumption (B), so for $\epsilon =\epsilon(Y)>0$ to be determined, assume that $x\in X$ and $r\in (0,\epsilon]$ satisfy $\mathcal{H}^{2n}(B(x,r))\geq (\omega_{2n}-\epsilon)r^{2n}$. Let $Z$ be any tangent cone of $X$ based at $x$. Because $X$ is an $RCD(0,2n)$ space, volume monotonicity then implies that $Z$ is metric cone with vertex $o_Z$ satisfying $\mathcal{H}^{2n-1}(\partial B(o_Z,1))\geq \mathcal{H}^{2n-1}(\mathbb{S}^{2n-1})-C(n)\epsilon$. By \cite[Theorem 15.80]{bamlerStructureTheoryNoncollapsed2021}, it follows that the entropy $W_{\infty}$ of $Z$ considered as a singular soliton satisfies 
$$W_{\infty} = \log \left( \frac{\mathcal{H}^{2n-1}(\partial B(o_Z,1))}{\mathcal{H}^{2n-1}(\mathbb{S}^{2n-1})} \right) \geq -C(n)\epsilon.$$
If $\epsilon =\epsilon(n)$ is sufficiently small, then \cite[Theorem 2.11]{bamlerStructureTheoryNoncollapsed2021} gives $x\in \mathcal{R}$. 
\end{proof}

\bibliographystyle{amsalpha}
\bibliography{klt}

\end{document}